\documentclass[11pt,a4paper]{article}
\usepackage{latexsym,amsfonts,amsmath,graphics}
\usepackage{epsfig}
\usepackage{enumitem}
\usepackage[usenames,dvipsnames,svgnames,table]{xcolor}

\usepackage{algorithmic}

\usepackage{graphicx}       

\usepackage[latin1]{inputenc}
\usepackage[T1]{fontenc}
\usepackage{microtype} 

\usepackage{cite}
\usepackage{array,colortbl}
\usepackage{amsmath,amsfonts,amssymb,bm,amsthm,mathtools,mathrsfs}
\DeclareFontFamily{U}{mathx}{\hyphenchar\font45}
\DeclareFontShape{U}{mathx}{m}{n}{
      <5> <6> <7> <8> <9> <10>
      <10.95> <12> <14.4> <17.28> <20.74> <24.88>
      mathx10
      }{}
\DeclareSymbolFont{mathx}{U}{mathx}{m}{n}
\DeclareFontSubstitution{U}{mathx}{m}{n}
\DeclareMathAccent{\widecheck}{0}{mathx}{"71}

\def\citep#1#2{\cite[{#1}]{#2}}

\usepackage{algorithm}
\usepackage{algorithmic}

\newcommand{\NN}{{\mathbb{N}}} 
\newcommand{\RR}{{\mathbb{R}}} 
\newcommand{\ZZ}{{\mathbb{Z}}} 

\DeclareSymbolFont{bbold}{U}{bbold}{m}{n}
\DeclareSymbolFontAlphabet{\mathbbold}{bbold}

\newcommand{\bsa}{{\boldsymbol{a}}}
\newcommand{\bsb}{{\boldsymbol{b}}}

\newcommand{\bsh}{{\boldsymbol{h}}}

\newcommand{\bsp}{{\boldsymbol{p}}}

\newcommand{\bsx}{{\boldsymbol{x}}}

\newcommand{\bszero}{{\boldsymbol{0}}} 

\newcommand{\bsgamma}{{\boldsymbol{\gamma}}}

\newcommand{\calB}{{\mathcal{B}}}

\newcommand{\calF}{{\mathcal{F}}}
\newcommand{\calG}{{\mathcal{G}}}
\newcommand{\calH}{{\mathcal{H}}}

\newcommand{\fraku}{{\mathfrak{u}}}

\newcommand{\icomp}{\mathrm{i}}
\newcommand{\abs}[1]{\left\vert#1\right\vert}
\newcommand{\norm}[1]{\left\Vert#1\right\Vert}
\newcommand{\rd}{\,\mathrm{d}} 

\usepackage[hidelinks,hypertexnames=false,hyperindex=true,pdfpagelabels,linktoc=all]{hyperref}
\usepackage{bookmark}
\makeatletter 
  \providecommand*{\toclevel@author}{999}
  \providecommand*{\toclevel@title}{0}
\makeatother

\usepackage{enumitem}

\theoremstyle{plain}
  \newtheorem{theorem}{Theorem}

\theoremstyle{definition}

\theoremstyle{remark}

\allowdisplaybreaks

\DeclareMathOperator{\err}{err}
\newcommand{\NOR}{{\mathrm{NOR}}}
\newcommand{\ABS}{{\mathrm{ABS}}}
\newcommand{\INT}{{\mathrm{INT}}}
\newcommand{\EMB}{{\mathrm{EMB}}}
\newcommand{\CRI}{{\mathrm{CRI}}}
\newcommand{\kor}{{\mathrm{kor}}}

\newcommand{\tlambda}{\widetilde{\lambda}}
\begin{document}

\title{Selected aspects of tractability analysis}

\author{Peter Kritzer}

\date{\today}

\maketitle

\begin{abstract}
    \noindent We give an overview of certain aspects of 
tractability analysis of multivariate problems. 
This paper is not intended to give a complete account of
the subject, but provides an insight into how the theory works for particular types of problems. 
We mainly focus on linear problems on Hilbert spaces, and mostly allow arbitrary linear information. 
In such cases, tractability analysis is closely linked to an analysis of the singular values of 
the operator under consideration.
We also highlight the more recent developments regarding exponential and generalized tractability. 
The theoretical results are illustrated by 
several examples throughout the article.
\end{abstract}


\tableofcontents

\section{Introduction}\label{sec:intro}

A typical situation studied in the field of \textit{Information-Based Complexity (IBC)} is 
the problem of approximating a mapping, frequently called the \textit{solution operator}, 
$S_d: \calF_d \to \calG_d$, where $(\calF_d,\norm{\cdot}_{\calF_d})$ and $(\calG_d,\norm{\cdot}_{\calG_d})$ are normed spaces. Here, 
the parameter $d$ is a positive integer; in typical examples like function approximation 
or numerical integration, where $\calF_d$ is a function space, $d$ corresponds to the number of 
variables the elements of $\calF_d$ depend on. This motivates why $d$ is usually called the \textit{dimension} 
of the problem. We stress, however, that we do not require $\calF_d$ to be a function space in this paper. 
In order to approximate $S_d (f)$ for an unknown $f\in\calF_d$, we use an \textit{algorithm} $A_{n,d}$, based on 
$n$ \textit{information} measurements, $L_1 (f), L_2 (f),\ldots,L_n (f)$, where the mappings 
$L_i$ are in a given class $\Lambda$ of information. The choice of the mappings $L_1 , L_2, \ldots,L_n$, as well as
that of $n$ are allowed to be adaptive,
i.e., $L_i(\cdot)= L_i (\cdot\,; L_1(f), L_2(f ), \ldots, L_{i-1} (f))$, and the number of measurements $n$ can be a function of the $L_i(f)$.
For a general introduction to the field of IBC, we refer to \cite{TWW88}. 

Measuring the error of an algorithm $A_{n,d}$ by a suitable error measure,
\[
 \err (A_{n,d})=\err (A_{n,d},S_d,\calF_d,\calG_d,\Lambda),
\]
it is natural to ask two questions:
\begin{itemize}
 \item For given dimension $d$, and a given error threshold $\varepsilon>0$, what is the amount 
 of information necessary to achieve an error of at most $\varepsilon$?
 \item How does the amount of information change if one varies $\varepsilon$ and/or the dimension $d$?
\end{itemize}
These questions are at the core of \textit{tractability analysis}, which can be considered a 
sub-field of IBC, and was started with the two papers \cite{W94a, W94b}. Since then, numerous 
papers have studied various numerical problems and their tractability, which led 
to the comprehensive three-volume book \cite{NW08}--\cite{NW12} on tractability of multivariate problems. 

It is the goal of this paper to highlight some aspects of classical tractability analysis, and partly also recent developments 
that have been published after the third book \cite{NW12} of the aforementioned trilogy. 
In order to outline the basic concepts, it is first necessary to make the problem setting more precise. Since 
we are interested in the amount of information required to achieve a certain error tolerance and the dependence 
of this amount on $d$, it is necessary to consider not only the problem of approximating $S_d:\calF_d\to\calG_d$
for some fixed $d$, but problems which are in fact 
a whole sequence of problems,
\[
 \{S_d\colon \calF_d\to\calG_d\}_{d\in\NN},
\]
where the $\calF_d$ and $\calG_d$ are normed spaces, and the $S_d$ are solution operators. Here and in the following, 
we denote by $\NN$ the set of positive integers. Again, we consider algorithms $A_{n,d}$, using $n$ information measurements, 
to approximate $S_d$. However, we are actually interested in the best one can do when using $n$ pieces of information, which 
results in the so-called \textit{$n$-th minimal error},
\[
 e_n (S_d):=\inf_{A_{n,d}} \err (A_{n,d}),
\]
where again $\err (\cdot)$ is a suitable error measure, and the infimum is extended over all admissible algorithms $A_{n,d}$. In 
later sections of this paper, we will be more specific about the choice of $\err (\cdot)$ as well as about which algorithms we consider.

We also define the \textit{initial error}, which usually is defined as 
the error of the best constant algorithm $A_{0,d,c}\equiv c$, where $c$ is 
a fixed element in $\calG_d$. In many settings commonly considered in IBC, 
it turns out that the initial error is attained by the zero algorithm, i.e. 
$A_{0,d,0}$ (see, e.g., \cite[Chapter 4]{NW08}). We shall denote the initital error by $e_0 (S_d)$ in the following.

The \textit{information complexity} $n(\varepsilon,S_d)$ is   
the minimal number $n$ of continuous linear functionals    
needed to find an algorithm $A_{n,d}$ that approximates   
$S_d$ with error at most $\varepsilon$. More precisely,    
we consider the absolute (ABS) and normalized (NOR) error criteria   
in which    
\begin{align*}   
 n(\varepsilon,S_d)&=n_{\ABS}(\varepsilon, S_d)\, =   
\min\{n\,\colon\,e_n (S_d)\le \varepsilon\},\\   
 n(\varepsilon,S_d)&=n_{\NOR}(\varepsilon,S_d)=   
\min\{n\,\colon\,e_n (S_d) \le \varepsilon\, e_0 (S_d)\}.   
\end{align*}   
In many applications, like in numerical integration or function approximation considered over 
certain reproducing kernel Hilbert spaces, it holds true that $e_0 (S_d)=1$, which means that 
the absolute and the normalized error criteria coincide. This, however, need not necessarily hold, and 
in such cases the initial error may have significant influence on the information complexity of a problem.

\medskip

The concept of \textit{tractability} of a problem refers to the situation when we can in a certain sense 
control the growth of the information complexity when $\varepsilon$ tends to zero or $d$ tends to infinity. 
Indeed, in the classical literature on IBC, a problem is called \textit{tractable} if the information complexity 
grows slower than exponentially in $\varepsilon^{-1}$ and in $d$. Otherwise, the problem is called \textit{intractable}. 
In particular, if the information complexity grows exponentially with $d$, we speak of the \textit{curse of dimensionality}. 
Hence, the curse of dimensionality implies intractability, but not necessarily vice versa. 

If a problem is tractable, one may consider 
several \textit{notions of tractability} by which we classify bounds on the growth of the information complexity. 

The most classic and at the same time most prominent notion of tractability is that of \textit{polynomial tractability (PT)}; 
indeed, we call a problem polynomially tractable in the setting $\CRI \in \{\ABS, \NOR\}$  if 
there exist absolute constants $C,p, q\ge 0$ such that
\begin{equation}\label{eq:def_PT}
n_{\CRI} (\varepsilon, S_d) \le C\, d^q\, \varepsilon^{-p}\quad \forall \varepsilon\in (0,1), \forall d\in\NN.
\end{equation}
It is important to note that the constants $C,p$, and $q$ are independent of $\varepsilon$ and $d$, 
and that \eqref{eq:def_PT} needs to hold for all choices of $\varepsilon$ and $d$ simultaneously for the definition to hold.

If \eqref{eq:def_PT} even holds for $q=0$, we speak of \textit{strong polynomial tractability (SPT)}; in this case the information 
complexity can be bounded independently of the dimension $d$. 

If SPT holds, i.e., if \eqref{eq:def_PT} holds with $q=0$, then the infimum of the $p$ for which this is the case is called the \textit{exponent of strong polynomial tractability}. It is also possible to study the best possible exponents $q$ and $p$ when we have polynomial tractability. However, in this case the optimal exponents are in general not uniquely defined, and we 
may decrease one at the expense of the other, which is why we restrict ourselves to presenting only exponents of strong polynomial tractability here. We refer to \cite[pp. 168--170]{NW08} for a more detailed discussion on trade-offs of the exponents. 

Another important notion of tractability is \textit{weak tractability (WT)}, which describes the boundary case to intractability. Indeed, 
a problem is weakly tractable in the setting $\CRI \in \{\ABS, \NOR\}$ if 
\begin{equation}\label{eq:def_WT}
  \lim_{d+\varepsilon^{-1}\to\infty} \frac{\log n_{\CRI} (\varepsilon, d)}{d+\varepsilon^{-1}}=0.
\end{equation}
If \eqref{eq:def_WT} is fulfilled, $n_{\CRI} (\varepsilon, d)$ does not depend exponentially on $d$ nor on $\varepsilon^{-1}$. 

There are indeed problems which are weakly tractable, but not polynomially tractable, such that the definition of weak 
tractability is justified; a simple example of such a problem is given when the information complexity behaves like $d^{\log d}$. 

The definition of WT has to be dealt with carefully, as the following example illustrates (cf. \cite[p. 7]{NW08}).  Consider a problem where 
\[
 n_{\CRI} (\varepsilon, d) = \mathrm{e}^{\sqrt{d}\,\sqrt{\varepsilon^{-1}}}.
\]
Then,
\[
 \lim_{d\to\infty} \frac{\log n_{\CRI} (\varepsilon, d)}{d + \varepsilon^{-1}} =0, \quad \mbox{and}\quad
 \lim_{\varepsilon^{-1}\to\infty} \frac{\log n_{\CRI} (\varepsilon, d)}{d + \varepsilon^{-1}}=0.
\]
However, we do not have WT, as for any $d\in\NN$ we can choose $\varepsilon_0=1/d$, and then 
\[
 \frac{\log n_{\CRI} (\varepsilon_0, d)}{d + \varepsilon_0^{-1}}=\frac{\sqrt{d}\, \sqrt{d}}{2d}=\frac{1}{2},
\]
which contradicts WT.

We 
further remark that there are various other notions of tractability in the literature, as 
for example \textit{quasi-polynomial tractability (QPT)} (see \cite{GW11} and below), $(s,t)$-weak tractability (see \cite{SW15}), 
or uniform weak tractability (see \cite{S13}). Note that there is a hierarchy between some (though not all) of these tractability notions. 
In particular, SPT implies PT, which in turn implies WT. We will discuss approaches to unify all tractability notions within one framework 
in Section \ref{sec:gen_tractability} below.

\medskip

Whether a given problem is tractable or not crucially depends on the problem settings, such as:
\begin{itemize}
 \item Is the problem considered with respect to the absolute or the normalized error criterion? While there are 
 of course examples where the initial error of a problem equals one, and therefore the absolute and the normalized error criteria coincide, 
 as for example function approximation or 
 numerical integration in certain function spaces like (weighted) Korobov spaces (see, e.g., \cite{NW08} and below), this is (by far) not
 true in general. Indeed, there are situations where the initial error of a problem depends exponentially on the dimension $d$, see, e.g.,
 \cite[Section 3.1.5]{NW08}, where the initial error of a problem from discrepancy theory tends to zero exponentially with increasing $d$, and therefore the problem in the absolute setting is uninteresting even for relatively small error thresholds. However, there exist also problems where the opposite occurs, and the initial error grows exponentially with $d$, see, e.g., \cite{KPW17}. 
 \item Which error measure is considered? In the literature on IBC, one can find numerous different error measures, and here we exemplary highlight 
 three very prominent ones; for instance, one may consider the \textit{worst case} error of an algorithm $A_{n,d}$, which is given by
 \begin{equation}\label{eq:worst-case-err}
   \err^{\rm wor} (A_{n,d}):=\sup_{f\in F} \norm{f-A_{n,d} (f)}_{\calG_d},
 \end{equation}
 where $F$ is a suitably chosen subset of $\calF_d$. Alternatively, algorithms are in some cases randomized by introducing a
 random element $\omega$, such that $A_{n,d}=A_{n,d,\omega}$. Then, one frequently studies the \textit{randomized error}, which 
 is defined as the worst case (again considered over $f\in F$ for $F\subseteq \calF_d$) with respect to an expected value of the error (see
 \cite[Section 4.3.3]{NW08} for a precise definition). Instead of considering the supremal error over all elements in a set $F\subseteq \calF_d$, 
 it is also possible to study the so-called \textit{average case error} of problems. Indeed, assume that each of the spaces $\{\calF_d\}_d$ is equipped 
 with a corresponding zero-mean Gaussian measure $\mu_d$. Then the \textit{average case error} of an algorithm $A_{n,d}$ is defined as
  \begin{equation}\label{eq:avg-case-err}
   \err^{\rm avg} (A_{n,d}):=\left(\int_{\calF_d} \norm{f-A_{n,d} (f)}_{\calG_d}^2 \mu (\rd f)\right)^{1/2}.
 \end{equation}
 \item What kind of information do we allow? It is essential to specify 
 the set $\Lambda$ which the mappings $L_i$ are elements of. The most common choices are $\Lambda=\Lambda^{\rm all}$, where all continuous linear functionals are allowed, and $\Lambda=\Lambda^{\rm std}$, the class of \textit{standard information}, where only function evaluations are allowed, provided that the $\calF_d$ are function spaces. A simple problem to illustrate the effect of switching the available information is 
 numerical integration: if we allow arbitrary linear information, an integration problem becomes trivial since we can choose $L_1 (f)$ equal 
 to the integral of $f$ and thus solve the problem exactly with only one 
 information mapping. If we only allow standard information, we need to analyze quadrature rules for numerical integration, which is much more involved. For other problems, as for example, function approximation, the situation is less clear, and there are cases in which there is little to no 
 difference in whether one uses $\Lambda^{\rm all}$ or $\Lambda^{\rm std}$. 
 Indeed, there has been considerable progress on the question of how powerful 
 $\Lambda^{\rm std}$ is in comparison to $\Lambda^{\rm all}$ for the problem 
 of $L_p$-approximation, especially in the cases $p=2$ and $p=\infty$. We exemplary refer to the papers \cite{DKU23, KU20,NSU22} for $L_2$-approximation on Hilbert spaces and to the papers \cite{DKU23, KPUU23, KU21} for $L_2$-approximation on general function classes (and the references therein).
 
 We also remark that there is work by various authors considering other types of information than just $\Lambda^{\rm all}$ or $\Lambda^{\rm std}$, such as for example absolute value information (cf. \cite{PSW20}).
 \end{itemize}

Before we proceed with the core parts of this overview article, we would like to recall from \cite[p. 149]{NW08} that---maybe counter-intuitively---a tractable problem is not necessarily easier to handle than an intractable problem. Indeed, suppose that we have a Problem 1 with
\[
  n_{\CRI}^{(\rm{P}1)} (\varepsilon, d) \simeq 10\, d^{10}\, \varepsilon^{-10}
\]
and a Problem 2 with
\[
  n_{\CRI}^{(\rm{P}2)} (\varepsilon, d) \simeq 1.01^d\, \varepsilon^{-1}.
\]
Obviously, Problem 1 is PT, whereas Problem 2 is intractable. However, for $\varepsilon=1/10$ and $d=10$, we have
$ n_{\CRI}^{(\rm{P}1)} (\varepsilon, d)\simeq 10^{21}$ and $n_{\CRI}^{(\rm{P}2)} (\varepsilon, d) \simeq 11$. Hence, from a practical point of view, Problem 2 will in many situations be easier to deal with than Problem 1.

\medskip

The rest of the paper is structured as follows. In the subsequent Section \ref{sec:example}, we illustrate some 
of the theory explained in the introduction by a concrete example defined on Korobov spaces. We shall return to this example 
repeatedly in this paper, namely in Sections~\ref{sec:example_2}, \ref{sec:example_3}, and \ref{sec:example_4}. 
In Section \ref{sec:Hilbert}, 
we give a first impression of tractability results for problems defined on sequences of Hilbert spaces, mostly with access to information from 
$\Lambda^{\rm all}$. In Section \ref{sec:exp_tractability}, we outline the concept of exponential tractability, and we explain 
generalized tractability in Section \ref{sec:gen_tractability}.

\section{Example: problems on Korobov spaces}\label{sec:example}

To illustrate the theoretical concepts and findings in this paper, we will consider a particular function space that 
has been studied in numerous papers on multivariate problems, and is of particular interest in quasi-Monte Carlo (QMC) integration, see, e.g., 
\cite{DKP22, DKS13, NW08}. 

The $d$-variate Korobov space $\calH_{\kor,d,\alpha}$ is a reproducing kernel Hilbert space of one-periodic functions on 
$[0,1]^d$, where periodicity is understood component-wise. For its definition, we need a real parameter $\alpha>1/2$ and 
a function $r_{2\alpha}: \ZZ\to \RR$, defined by
\[
 r_{2\alpha}(h):=\begin{cases}
                  1 & \mbox{if $h=0$,}\\
                  \abs{h}^{2\alpha} & \mbox{if $h\neq 0$,}
                 \end{cases}
\]
for any $h\in\ZZ$. We remark that some authors write $\alpha$ instead of $2\alpha$ in the previous definition, with the 
according adaptions in all other definitions to follow. The reason why we choose to use $2\alpha$ is that then, for integer choices 
of $\alpha$, this parameter is directly linked to the number of existing square integrable derivatives of the elements of the space. 

We also need a multivariate version of this function, denoted by $r_{d,2\alpha}$ for $d\in\NN$, given as
\[
 r_{d,2\alpha}(\bsh):=\prod_{j=1}^d r_{2\alpha} (h_j), 
\]
for any $\bsh=(h_1,\ldots,h_d) \in \ZZ^d$.
For a function $f\in L_2 ([0,1]^d)$, we consider its representation by its Fourier series,
\[
 f(\bsx)\sim\sum_{\bsh\in\ZZ^d} \widehat{f} (\bsh) \mathrm{e}^{2\pi\icomp\bsh\cdot\bsx}\quad \mbox{for}\ \bsx\in [0,1]^d,
\]
where the $\bsh$-th Fourier coefficient for $\bsh\in\ZZ^d$ is given by 
\[
 \widehat{f}(\bsh):=\int_{[0,1]^d} f(\bsx) \mathrm{e}^{-2\pi\icomp\bsh\cdot\bsx} \rd \bsx.
\]
The Korobov space $\calH_{\kor,d,\alpha}$ is a subspace of $L_2 ([0,1]^d)$ and is defined as the set of 
all one-periodic functions with absolutely convergent Fourier series and a finite norm $\norm{f}_{\kor,d,\alpha}
=\langle f, f\rangle_{\kor,d,\alpha}^{1/2}$, where the inner product is given by 
\[
 \langle f, g\rangle_{\kor,d,\alpha}=\sum_{\bsh\in\ZZ^d} r_{d,2\alpha} (\bsh) \widehat{f}(\bsh) \overline{\widehat{g} (\bsh)}.
\]
As mentioned before, the parameter $\alpha$ in the above definitions is directly related to the number of 
partial derivatives of the elements $f\in \calH_{\kor,d,\alpha}$, which is why $\alpha$ is commonly 
referred to as the smoothness parameter of the space (see, e.g., \cite[Chapter 2]{DKP22} for details). 

For later considerations, we also need to determine an orthonormal basis of $\calH_{\kor,d,\alpha}$. 
It is easily checked that the functions $\eta_{\bsh}$, $\bsh\in\ZZ^d$, with
\begin{equation}\label{eq:onb_korobov}
 \eta_{\bsh}(\bsx) (r_{d,2\alpha}(\bsh))^{-1/2} \mathrm{e}^{-2\pi\icomp\bsh\cdot\bsx}=(r_{d,\alpha} (\bsh))^{-1} \mathrm{e}^{-2\pi\icomp\bsh\cdot\bsx}
 \quad \mbox{for $\bsx\in [0,1]^d$,}
\end{equation}
form an orthonormal basis. 

It is also helpful to note that $\calH_{\kor,d,\alpha}$ is actually a tensor product space (see, e.g., \cite{DKP22, NW08}). 
Indeed, $\calH_{\kor,d,\alpha}$ is the $d$-fold tensor product of the univariate Korobov space $\calH_{\kor,1,\alpha}$, 
\begin{equation}\label{eq:kor_tensor}
 \calH_{\kor,d,\alpha} =
 \underbrace{\calH_{\kor,1,\alpha} \otimes \calH_{\kor,1,\alpha} \otimes \cdots \otimes \calH_{\kor,1,\alpha}}_{\mbox{$d$ times}},
\end{equation}
which is to be understood as the closure (with respect to the norm $\norm{\cdot}_{\kor,d,\alpha}$) of the span 
of all functions $\prod_{j=1}^d f_j$, with all $f_j \in \calH_{\kor,1,\alpha}$.

\medskip

We may study a multivariate problem on $\calH_{\kor,d,\alpha}$, expressed as the problem of 
approximating an operator $S_d:\calH_{\kor,d,\alpha}\to \calG_d$ for some suitable choice of 
a normed space $\calG_d$. Indeed, we can do this for any $d\in\NN$, and thus study the whole 
sequence of problems $\{S_d:\calH_{\kor,d,\alpha}\to\calG_d\}_{d\in\NN}$, and its tractability properties. 
Below, we will deal with two particular choices of $S_d$ for this example, namely
\begin{itemize}
 \item The case when $\calG_d=L_2([0,1]^d)$ for $d\in\NN$ and $S_d(f)=\EMB_d(f)=f$ is the embedding 
 from $\calH_{\kor,d,\alpha}$ to $L_2([0,1]^d)$, i.e., $L_2$-approximation of functions,
 \item the case when $\calG_d=\RR$, and $S_d(f)=\INT_d (f)=\int_{[0,1]^d}f(\bsx)\rd \bsx$, i.e., numerical integration. 
\end{itemize}
For the first case, when we study $L_2$-approximation, we can allow either $\Lambda^{\rm all}$ or $\Lambda^{\rm std}$ as the information class. 
Obviously, the problem is more challenging if we restrict ourselves to the latter. It is easy to see that for this problem the initial error equals one, 
so the absolute and the normalized settings coincide.

For the integration problem, if we allow $\Lambda^{\rm all}$, we may just choose an algorithm $A_{n,d} (f) = \INT_d (f)$ for $n=1$ and obtain 
that $n_{\CRI}(\varepsilon,d)=1$ and tractability. If we only allow $\Lambda^{\rm std}$, the problem becomes much harder and there 
is a huge literature on numerically integrating elements of $\calH_{\kor,d,\alpha}$ by quasi-Monte Carlo (QMC) rules, see, among many others, 
\cite{DKP22,DKS13}. Note, however, that in the case of $\Lambda^{\rm std}$ approximating $\INT_d$ is a problem not harder than approximating 
$\EMB_d$. For a proof of this result, we refer to \cite{NSW04}. We further 
remark that also for the integration problem the initial error equals one.

We shall return to this example again in Sections~\ref{sec:example_2},~\ref{sec:example_3}, and~\ref{sec:example_4}. In general, 
positive results for integration problems can often be obtained by providing suitable integration rules. For deriving lower bounds on 
the integration error or the information complexity, a common technique is to use so-called bump functions, which are well chosen functions
that yield a particularly high error for concrete integration rules. We also would like to mention that 
recently a further technique for showing lower bounds for the integration of periodic functions has been introduced, 
that is based on a modification of the Schur product theorem for matrices. We refer to \cite{KV23} for details on 
these methods.

\section{Tractability for linear problems on Hilbert spaces}\label{sec:Hilbert}

Very generally speaking, we have several options to determine whether a given 
problem is tractable or not: 
\begin{itemize}
 \item By showing upper error bounds for particular algorithms.
 \item By showing lower error bounds for all admissible algorithms (which usually is much harder than showing upper bounds for concrete algorithms).
 \item By checking criteria which are equivalent to selected tractability notions and maybe easier to verify, which works for special settings. 
 Such conditions may be formulated, e.g., in terms of the properties of the operators $S_d$ and the spaces $\calF_d, \calG_d$, as will be shown below. 
\end{itemize}
In this section, we will treat a setting where we can proceed like in the last point from above, and we will show exemplary how criteria equivalent to certain tractability notions can be derived. 

\subsection{Criteria for tractability}

The setting we consider here is based on the assumption that the 
$\{\calF_d\}_{d\in\NN}$ and $\{\calG_d\}_{d\in\NN}$ are Hilbert spaces, and 
we let $\{S_d : \calF_d \to \calG_d\}_{d \in \NN}$ be a sequence of
compact linear solution operators with adjoint operators $S_d^*$ such that $S_d^*S_d : \calF_d \to \calF_d$ has eigenvalues and orthonormal eigenvectors
\begin{equation} \label{eq:eigenvalues}
 \lambda_{1,d} \ge \lambda_{2,d} \ge \cdots \ge 0, \qquad u_{1,d}, u_{2,d}, \ldots, \qquad d \in \NN.
\end{equation}
This means that the $\sqrt{\lambda_{i,d}}$ are the singular values of the operator $S_d$. The assumption that the $S_d$ are compact implies that $\lambda_{i,d}$ converges to zero as $i$ tends to infinity for every $d$, which in turn implies that the problem is solvable by suitable algorithms. Indeed, compactness of $S_d$ is equivalent to $n_{\CRI} (\varepsilon, d) <\infty$ for all $\varepsilon>0$ (see, e.g., \cite[Section 4.2.3]{NW08} for further details). 

Let us now assume that we have access to information from $\Lambda^{\rm all}$, i.e., we allow the mappings $L_i$ to be arbitrary continuous linear functionals in 
$\calF^*$. Furthermore, we also allow the algorithms to be adaptive. We would like to approximate $S_d$ by algorithms $A_{n,d}$, and study the worst case error 
over the unit ball $\calB_d$ in $\calF_d$, i.e., we consider
\[
 e_n (S_d)=\inf_{A_{n,d}}\sup_{f\in\calB_d} \norm{S_d(f)-A_{n,d}(f)}_{\calG_d}.
\]
In this setting, the optimal approximate solution operator is known to be
\[
A_{n,d}^{\rm opt}(f) = \sum_{i=1}^n S_d(u_{i,d}) \langle f, u_{i,d}\rangle_{\calF_d},
\]
where $\langle \cdot, \cdot \rangle_{\calF_d}$ is the inner product in $\calF_d$. 
Furthermore, it is known that $\err^{\rm wor}(A_{n,d}^{\rm opt})=\sqrt{\lambda_{n+1,d}}$,
and so
\begin{equation}\label{eq:infcomp_Hilbert}
 n_{\ABS} (\varepsilon,d) = \min \{n \in \NN_0\colon \sqrt{\lambda_{n+1,d}} \le \varepsilon\},
\end{equation}
see \cite[Section 5.1]{NW08}. 
For the sake of brevity and clarity of notation, we restrict ourselves to considering only the absolute error criterion in this context.
If we alternatively would consider the normalized error criterion, we would divide the error by the initial error, which is the operator norm of $S_d$, 
which is $\sqrt{\lambda_{1,d}}$ in the current setting. I.e., in the normalized setting the information complexity is given by
\begin{equation}\label{eq:infcomp_Hilbert_NOR}
 n_{\NOR} (\varepsilon,d) = \min \{n \in \NN_0\colon \sqrt{\lambda_{n+1,d}} \le \varepsilon\, \sqrt{\lambda_{1,d}}\}.
\end{equation}
This means that, essentially, in the tractability analysis in this section, all criteria would be normalized by the first singular value $\lambda_{1,d}$. 

Let us, however, continue with the absolute setting to keep notation simple. Due to \eqref{eq:infcomp_Hilbert}, it is natural to study the 
decay of the eigenvalues $\lambda_{n,d}$ if we would like to derive necessary and sufficient conditions for various tractability notions, 
which has been done extensively in the literature on IBC. To give an example, the following theorem is due to Wo\'{z}niakowski (see \cite{W94a}).
Since the proof of the following theorem is typical for such results, we provide the key steps of the proof (for a version including all 
technical details, see, e.g., \cite[Proof of Theorem 5.1]{NW08}).
\begin{theorem}[Wo\'{z}niakowski]\label{thm:ALG_SPT}
 Let $\{\calF_d\}_{d\in\NN}$ and $\{\calG_d\}_{d\in\NN}$ be Hilbert spaces, and let 
$\{S_d\colon\calF_d\to \calG_d\}_{d\in\NN}$ be compact linear operators. Consider information from $\Lambda^{\rm all}$ and 
the absolute worst case setting on the unit balls $\calB_d$ of the $\calF_d$. 

Then we have SPT if and only if there exists a constant $\tau>0$ and a constant $L\in\NN$ such that
\begin{equation}\label{eq:cond_SPT}
 M:=\sup_{d\in\NN} \left(\sum_{n=L}^\infty \lambda_{n,d}^\tau \right)^{1/\tau} <\infty. 
\end{equation}
The exponent of SPT is given by
\[
 p^*:=\inf\{2\tau\colon \tau \ \mbox{satisfies \eqref{eq:cond_SPT}}\}.
\]
\end{theorem}
\begin{proof}[Sketch of the proof of Theorem \ref{thm:ALG_SPT}]
 The main steps in the proof are as follows. 
 
 We first show necessity of \eqref{eq:cond_SPT}. Indeed, assume that we have SPT. I.e., there 
 exist absolute constants $C,p\ge 0$ such that
 \[
  n_{\ABS} (\varepsilon,d)\le C\, \varepsilon^{-p} \quad \forall\varepsilon\in (0,1), \forall d\in\NN.
 \]
 We can assume that $C\ge 1$ and that $p=p^*+\delta$, where $p^*$ is the exponent of SPT and $\delta$ is some (small) positive number.
 By \eqref{eq:infcomp_Hilbert}, we know that 
 \[
  \lambda_{n_{\ABS} (\varepsilon,d)+1,d} \le \varepsilon^2,
 \]
 and as the sequence of the $\lambda_{n,d}$ is non-increasing, it follows that 
 \[
  \lambda_{\lfloor C\, \varepsilon^{-p}\rfloor+1,d} \le \varepsilon^2.
 \]
 Next, define $n(\varepsilon):=\lfloor C\, \varepsilon^{-p}\rfloor+1$ for $\varepsilon\in (0,1)$. When 
 we vary the value of $\varepsilon\in (0,1)$, then $n(\varepsilon)$ takes on the values $\lfloor C\rfloor +1$, 
 $\lfloor C\rfloor +2$, $\lfloor C\rfloor + 3$, \ldots .

 On the other hand, it is obviously true that $n(\varepsilon)\le C\, \varepsilon^{-p}+1$, which is equivalent to
 \[
  \varepsilon^2 \le \left(\frac{C}{n(\varepsilon)-1}\right)^{2/p}, 
 \]
 which yields
 \[
  \lambda_{n(\varepsilon),d}=\lambda_{\lfloor C\, \varepsilon^{-p}\rfloor+1,d}\le \varepsilon^2 \le \left(\frac{C}{n(\varepsilon)-1}\right)^{2/p}.
 \]
 This observation holds for any $\varepsilon\in (0,1)$, and by varying $\varepsilon$ (and thereby also $n(\varepsilon)$), we obtain
 \[
  \lambda_{n,d}\le \left(\frac{C}{n-1}\right)^{2/p},\quad \forall n\ge \lfloor C \rfloor +1.
 \]
 From this, it is not hard to derive the condition \eqref{eq:cond_SPT} by choosing $L=\lfloor C \rfloor +1 \ge 2$ and $\tau>p/2$. Since $p=p^*+\delta$, we can choose $\tau$ arbitrarily close to $p^*/2$. 
 
 \medskip
 
 For the converse, assume now that \eqref{eq:cond_SPT} holds for some $\tau>0$ and some $L\in\NN$. 
 Due to the ordering of the eigenvalues we have, for any $n\ge L$, that
 \[
  (n-L+1)\, \lambda_{n,d} \le \sum_{i=L}^n \lambda_{i,d} \le \sum_{i=L}^\infty \lambda_{i,d},
 \]
 and with analogous reasoning
 \[
  (n-L+1)^{1/\tau}\, \lambda_{n,d} \le \left(\sum_{i=L}^n \lambda_{i,d}^\tau\right)^{1/\tau} \le \left(\sum_{i=L}^\infty \lambda_{i,d}^\tau\right)^{1/\tau} \le M.  
 \]
 
 Now, choose the smallest $n_0\ge L$ such that $(n_0-L+1)^{-1/\tau} \, M \le \varepsilon^2$, which guarantees that
 $\lambda_{n_0+1,d}\le \varepsilon^2$, and this implies
 \[
  n_{\ABS} (\varepsilon,d) \le n_0.
 \]
 From the latter inequality and from the choice of $n_0$, one can deduce SPT by elementary algebra. Furthermore, it is easy to see that the exponent of SPT is then 
 at most $\inf\{2\tau\colon \tau \ \mbox{satisfies \eqref{eq:cond_SPT}}\}$. Combining the observations on the relation between $p^*$ and $\tau$ in the two steps of the proof shows that the claim on the exponent of SPT holds true. 
\end{proof}
Results of a similar flavor as Theorem \ref{thm:ALG_SPT} exist for many other 
tractability notions; for instance, we have the following theorem on polynomial tractability, which is also due to 
Wo\'{z}niakowski (see \cite{W94a}).
\begin{theorem}[Wo\'{z}niakowski]\label{thm:ALG_PT}
 Let $\{\calF_d\}_{d\in\NN}$ and $\{\calG_d\}_{d\in\NN}$ be Hilbert spaces, and let 
$\{S_d\colon\calF_d\to \calG_d\}_{d\in\NN}$ be compact linear operators. Consider information from $\Lambda^{\rm all}$ and 
the absolute worst case setting on the unit balls $\calB_d$ of the $\calF_d$. 

Then we have PT if and only if there exist constants $\tau_1,\tau_2\ge 0$ and $\tau_3,H>0$ such that
\begin{equation}\label{eq:cond_PT}
 M:=\sup_{d\in\NN}\, d^{-\tau_1}\, \left(\sum_{n=\lceil H d^{\tau_2}\rceil}^\infty \lambda_{n,d}^{\tau_3} \right)^{1/\tau_3} <\infty. 
\end{equation}
\end{theorem}
For a proof of Theorem \ref{thm:ALG_PT}, we again refer to \cite[Proof of Theorem 5.1]{NW08}. Indeed, Theorem \ref{thm:ALG_SPT} is a special case of Theorem~\ref{thm:ALG_PT} and the proof of the latter automatically yields a proof of the former, by setting $\tau_1=\tau_2=0$ in Theorem~\ref{thm:ALG_PT} and choosing $L$ accordingly in Theorem \ref{thm:ALG_SPT}. 
As a further example, we also state a theorem with an equivalent condition for WT due to Werschulz and Wo\'{z}niakowski (see \cite{WW17}, to which we also refer for a proof). 
\begin{theorem}[Werschulz, Wo\'{z}niakowski]\label{thm:ALG_WT}
Let $\{\calF_d\}_{d\in\NN}$ and $\{\calG_d\}_{d\in\NN}$ be Hil\-bert spaces, and let 
$\{S_d\colon\calF_d\to \calG_d\}_{d\in\NN}$ be compact linear operators. Consider information from $\Lambda^{\rm all}$ and 
the absolute worst case setting on the unit balls $\calB_d$ of the $\calF_d$. 

We have WT if and only if 
\begin{equation}\label{eq:cond_WT}
 \sup_{d\in\NN} e^{-cd} \sum_{n=1}^\infty \mathrm{e}^{-c\,\lambda_{n,d}^{-1/2}} <\infty\quad \forall c>0. 
\end{equation}
\end{theorem} 

For an overview of conditions that are equivalent to the various common 
tractability notions, we refer to \cite{KW19}.

\subsection{Linear tensor product problems}

A special case of the Hilbert space setting occurs when one considers so-called \textit{linear tensor product problems}. 
Indeed, consider two Hilbert spaces $\calF_1$ and $\calG_1$ and a 
compact linear solution operator,
$S_1: \calF_1 \to \calG_1$.
For $d\in \NN$, let 
$$
\calF_d\,=\,\calF_1\otimes\calF_1\otimes\cdots\otimes\calF_1\quad\mbox{and}\quad
\calG_d\,=\,\calG_1\otimes\calG_1\otimes\cdots\otimes\calG_1
$$
be the $d$-fold tensor products of the spaces $\calF_1$ and $\calG_1$, respectively. 
Furthermore, let $S_d$ be the linear tensor product operator,
$$ 
S_d=S_1\otimes S_1\otimes \cdots \otimes S_1,
$$
on $\calF_d$. 
In this way, we again obtain a sequence of compact    
linear solution operators    
$$   
\{S_d: \calF_d \to \calG_d\}_{d \in \NN}.     
$$
It is known that the eigenvalues $\lambda_{n,d}$ of $S_d^*S_d$ are then given as products 
of the eigenvalues $\tlambda_n$ of the operator $S_1^\ast S_1:\calF_1\rightarrow\calF_1$, i.e.,
\begin{equation}\label{eq:productform}
\lambda_{n,d}=\prod_{\ell=1}^d \tlambda_{n_\ell}.
\end{equation}
Without loss of generality, one can assume that the $\tlambda_n$ are ordered, i.e., $\tlambda_1\ge \tlambda_2 \ge \cdots$. 
Tractability analysis in the tensor product setting thus can be done by considering properties of the eigenvalues $\tlambda_n$ 
of the operator $S_1^\ast S_1$. Deriving necessary and sufficient conditions on the $\tlambda_n$ poses an interesting 
mathematical problem. However, from this point of view, tractability analysis for tensor product problems is not necessarily 
simpler than for the general case. Indeed, although the $\lambda_{n,d}$ are given by \eqref{eq:productform}, 
the ordering of the $\tlambda_n$ does not easily imply the
ordering of the $\lambda_{n,d}$ since the map $n \in \NN \mapsto (n_1, \ldots, n_d) \in \NN^d$ corresponding to \eqref{eq:productform}
exists but usually does not have a simple explicit form. This makes the tractability analysis challenging. Numerous results on precise 
conditions for common tractability notions in the tensor product case can be found in \cite[Chapter 5]{NW08}, but we would like to state one exemplary result, which was again first shown by Wo\'{z}niakowski in \cite{W94b}, and further elaborated in \cite[Theorem 5.5]{NW08}.
\begin{theorem}[Wo\'{z}niakowski]\label{thm:ALG_PT_tensor}
 Let $\{\calF_d\}_{d\in\NN}$ and $\{\calG_d\}_{d\in\NN}$ be Hilbert spaces, and let 
$\{S_d\colon\calF_d\to \calG_d\}_{d\in\NN}$ be compact linear operators. Consider information from $\Lambda^{\rm all}$ and 
the absolute worst case setting on the unit balls $\calB_d$ of the $\calF_d$. 

Then the following assertions hold true.
\begin{itemize}
 \item The problem is intractable if $\widetilde{\lambda}_1>1$. 
 \item If $\widetilde{\lambda}_1=\widetilde{\lambda}_2=1$, the problem is intractable. 
 \item If $\widetilde{\lambda}_1=1$ and $\widetilde{\lambda}_2 < 1$, we 
 cannot have PT, but we may have WT, depending on the decay of the 
 $\widetilde{\lambda}_n$.
 \item If $\widetilde{\lambda}_1 < 1$, then we may have WT and also PT and SPT, 
 provided that the $\widetilde{\lambda}_n$ decay sufficiently fast. 
\end{itemize}
\end{theorem}

\subsection{The average case setting}

At this point let us move away from tensor product problems again, and let us make a short detour to the average case setting. As pointed out in the introduction, 
in this case we assume that the spaces $\{\calF_d\}_{d\in\NN}$ are equipped with zero-mean Gaussian measures 
$\mu_d$, and we let $\nu_d:= \mu_d S_d^{-1}$ be the corresponding zero-mean Gaussian measures on the spaces 
of solution elements $S_d (f)$ for $d\in\NN$. Then, one considers the so-called \textit{correlation operators}
$C_{\nu_d}$ and its eigenvalues, which we denote by $\lambda_{n,d}^{\rm avg}$, and of which we assume that 
they are ordered, $\lambda_{1,d}^{\rm avg}\ge\lambda_{2,d}^{\rm avg}\ge \cdots \ge 0$. It is known that we then have, for the $n$-th minimal 
average case error $e_n^{\rm avg} (S_d)$,
\[
 \err_n^{\rm avg} (S_d)=\left(\sum_{k=n+1}^\infty \lambda_{k,d}^{\rm avg}\right)^{1/2},
\]
and
\[
 n_{\ABS} (\varepsilon, S_d)=\min\left\{n\colon \sum_{k=n+1}^\infty \lambda_{k,d}^{\rm avg} \le \varepsilon^2\right\}.
\]
If one would like to consider the normalized setting instead of the absolute setting, this can be done analogously by noting 
that the initial average case error, for any $d\in\NN$, equals the square root of the (finite) trace of the covariance operator $C_{\nu_d}$,
\[
 \err_0^{\rm avg} (S_d)=\left(\sum_{k=1}^\infty \lambda_{k,d}^{\rm avg}\right)^{1/2}.
\]
Due to these representations, it is possible to find necessary and sufficient conditions on tractability by considering 
the tail sums $\sum_{k=n+1}^\infty \lambda_{k,d}^{\rm avg}$ in the average case setting, instead of considering the $\lambda_{n,d}$ in 
the worst case setting. 

We refer to \cite[Chapter 6]{NW08} for results of this kind, and also to \cite[Appendix B]{NW08} for technical background 
on Gaussian measures and correlation operators. We also remark that such results for the average case setting can be shown 
without having to assume that the spaces $\{\calF_d\}_{d\in\NN}$ are Hilbert spaces, but it is sufficient to assume 
that they are separable Banach spaces; however, the $\{\calG_d\}_{d\in\NN}$ are still assumed to be Hilbert spaces. 

Analyzing the eigenvalues is a very common way of dealing with tractability of linear problems on Hilbert spaces. Alternatively, one may 
work with the concepts of the diameter and/or radius of information 
(see, e.g., \cite[Chapter 4]{NW08}), which are rather technical but very powerful tools in this context. Furthermore, one can also derive interesting results by relating a problem for $\Lambda^{\rm all}$ to the analogous 
problem for $\Lambda^{\rm std}$, where we would like to mention again the papers \cite{DKU23, KPUU23, KU20, KU21, NSU22} as recent examples. 

\section{Example: problems on Korobov spaces, Part 2}\label{sec:example_2}

Let us return to the example of the Korobov spaces defined in Section~\ref{sec:example}, and show some results for the worst case setting.

\medskip

First, we study $L_2$-approximation, i.e., $S_d:\calH_{\kor,d,\alpha}\to L_2 ([0,1]^d)$, $S_d (f)=\EMB_d (f)=f$. 
Let us allow $\Lambda^{\rm all}$ as the information class. Then, we need to identify the 
eigenvalues of the self-adjoint operator $\EMB_d^* \EMB_d$, which is (see, e.g., \cite{NW08}) given by
\[
 (\EMB_d^* \EMB_d (f))(\bsx)=\sum_{\bsh\in\ZZ^d} r_{d,2\alpha}(\bsh)\, \langle f,\eta_{\bsh}\rangle_{\kor,d,\alpha}\, \eta_{\bsh} (\bsx), 
 \quad \mbox{for $\bsx\in [0,1]^d$,}
\]
where the $\eta_{\bsh}$ are given in \eqref{eq:onb_korobov}. Then, the eigenpairs of $\EMB_d^* \EMB_d$ 
are $((r_{d,2\alpha} (\bsh))^{-1},\eta_{\bsh})$, since
\[
 \EMB_d^*\, \EMB_d (\eta_{\bsh})=(r_{d,2\alpha} (\bsh))^{-1}\,\eta_{\bsh}.
\]
Hence, for this concrete problem, we have
\[
 \left\{\lambda_{n,d}\colon n\in\NN\right\}=\left\{(r_{d,2\alpha} (\bsh))^{-1}\colon \bsh\in\ZZ^d\right\}.
\]
Note that it is not straightforward to order the values of the $(r_{d,2\alpha} (\bsh))^{-1}$, but since the results in Theorems \ref{thm:ALG_SPT}--\ref{thm:ALG_WT} are formulated in terms of infinite sums of the eigenvalues, a precise ordering is not needed here. Indeed, we are going to show next that
none of the conditions in these theorems is met, so the problem is not even weakly tractable. To this end, 
consider the condition \eqref{eq:cond_WT} for arbitrary $d\in\NN$ and for the special choice $c=c_0=1/2$. 
Then we have
\begin{eqnarray*}
 \sum_{n=1}^{\infty} \mathrm{e}^{-c_0\, \lambda_{n,d}^{-1/2}}&=& 
 \sum_{\bsh\in\ZZ^d} \mathrm{e}^{-c_0\, r_{d,\alpha}(\bsh)}\\
 &=& \left(\sum_{h\in\ZZ} \mathrm{e}^{-c_0\, r_{\alpha}(h)}\right)^d\\
 &=& \left(1+ 2\sum_{h=1}^\infty \mathrm{e}^{-c_0\, h^{\alpha}}\right)^d\\
 &\ge &\left(1+ 2 \mathrm{e}^{-c_0}\right)^d.
\end{eqnarray*}
Consequently,
\[
\mathrm{e}^{-c_0\, d} \sum_{n=1}^{\infty} \mathrm{e}^{-c_0\, \lambda_{n,d}^{-1/2}}
\ge  \mathrm{e}^{d \left(-c_0 + \log (1+2 \mathrm{e}^{-c_0})\right)}. 
\]
However, it is easily checked that for the special choice $c_0=1/2$ we have 
$\left(-c_0 + \log (1+2 \mathrm{e}^{-c_0})\right)>0$, and this implies that 
the condition \eqref{eq:cond_WT} is not satisfied for $c=1/2$. Hence, we cannot have WT, so the problem 
of $L_2$-approximation on $\calH_{\kor,d,\alpha}$ is intractable.

There is also a second, actually shorter, way to see that $L_2$-approximation on $\calH_{\kor,d,\alpha}$ is intractable, 
namely by considering it as a tensor product problem. Indeed, recall from \eqref{eq:kor_tensor} that $\calH_{\kor,d,\alpha}$ is the $d$-fold tensor product of the spaces 
$\calH_{\kor,1,\alpha}$. Accordingly, following what is outlined above 
for linear tensor product problems on Hilbert spaces, we can analyze the eigenvalues $\widetilde{\lambda}_n$ of the operator $S_1=\EMB_1$. 
In analogy to the eigenvalues for the $d$-variate problem, these are given by the set of the $(r_{1,2\alpha} (h))^{-1}$
for $h\in\ZZ$. It follows from the definition of $r_{1,2\alpha}$ that the largest eigenvalues are $r_{1,2\alpha}(0)$, 
$r_{1,2\alpha}(1)$, and $r_{1,2\alpha}(-1)$, which all equal $1$. Hence it follows by the second point of Theorem \ref{thm:ALG_PT_tensor} 
that the problem is intractable.

\medskip

What happens if we study, instead of $L_2$-approximation, the problem of numerical integration on $\calH_{\kor,d,\alpha}$, i.e., 
if we choose $\calG_d=\RR$ and $S_d =\INT_d$ for $d\in\NN$? As pointed out above, in this case it makes sense to restrict ourselves 
to standard information, i.e., $\Lambda=\Lambda^{\rm std}$. We have pointed out in Section \ref{sec:example} that integration 
on $\calH_{\kor,d,\alpha}$ is not harder than $L_2$-approximation, so one might hope that maybe the integration problem could 
satisfy some sort of tractability. However, it is known (see, e.g., \cite{HW01, SW01} and also \cite[Chapter 16]{NW10}) 
that also numerical integration on $\calH_{\kor,d,\alpha}$ is intractable. 

\medskip

The situation that both the $L_2$-approximation problem and the numerical integration problem on $\calH_{\kor,d,\alpha}$ are intractable 
is unsatisfying. However, by modifying the spaces $\calH_{\kor,d,\alpha}$ to so-called \textit{weighted function spaces}, 
also positive results can be obtained. The idea of studying weighted function spaces goes back to the seminal paper \cite{SW98} of 
Sloan and Wo\'{zniakowski}. The motivation for weighted spaces is that in many applications different coordinates or different groups of 
coordinates may have different influence on a multivariate problem. To give a simple example, consider a function $f:[0,1]^d\to\RR$, where 
\[
 f(x_1,\ldots,x_d)=\mathrm{e}^{x_1}+ \frac{x_2 + \cdots +x_d}{2^d}.
\]
Clearly, for large $d$, the first variable has much more influence on this problem than the others. 
In order to make such observations more precise, one introduces weights, which are nonnegative real numbers 
$\gamma_{d,\fraku}$, one for each set $\fraku \subseteq \{1,\ldots,d\}$. Intuitively speaking, the number $\gamma_{d,\fraku}$ models the influence 
of the variables with indices in $\fraku$. Larger values of $\gamma_{\fraku}$ mean more influence, smaller values less influence. Formally, we
set $\gamma_{d,\emptyset}=1$, and we write $\bsgamma_d=\{\gamma_{d,\fraku}\}_{\fraku\subseteq \{1,\ldots,d\}}$. These weights can now be used 
to modify the norm in a given function space, thereby modifying the unit ball over which the worst case error of a problem is considered. By making 
the unit ball smaller according to the weights (in the sense that also here certain groups of variables may have less influence than others) a problem 
may thus become tractable, provided that suitable conditions
on the weights hold. This effect also corresponds to intuition---if a problem depends
on many variables, of which only some have significant influence, it is natural to
expect that the problem will be easier to solve than one where all variables have the
same influence.

For the mathematically precise definition of the weighted Korobov space, we restrict ourselves to the situation where all weights 
$\bsgamma_d=\{\gamma_{d,\fraku}\}_{\fraku\subseteq \{1,\ldots,d\}}$, 
are strictly positive. The more general case in which zero weights are allowed is dealt with by making suitable technical adaptions, 
which is possible but slightly tedious. 
Furthermore, for $\bsh=(h_1,\ldots,h_d)\in\ZZ^d$, we put
\[
 \fraku (\bsh):= \{j\in\{1,\ldots,d\}\colon h_j \neq 0\},
\]
and modify the function $r_{d, 2\alpha}$ to 
\[
 r_{d,2\alpha,\bsgamma_d} (\bsh):=\frac{1}{\gamma_{d,\fraku (\bsh)}} \prod_{j\in\fraku (\bsh)} \abs{h_j}^{2\alpha}\quad 
 \mbox{for $\bsh=(h_1,\ldots,h_d)\in\ZZ^d$,}
\]
where we define the empty product to equal 1 if $\bsh=\bszero$.

For $\alpha>1/2$ we define the Hilbert space $\calH_{\kor,d,\alpha,\bsgamma_d}$ 
as the space of all one-periodic functions $f$ with absolutely convergent Fourier series, 
and with finite norm $\norm{f}_{\kor,d,\alpha,\bsgamma_d}:=\langle f, f\rangle_{\kor,d,\alpha,\bsgamma_d}^{1/2}$
where the inner product is given by
\[
 \langle f, g\rangle_{\kor,d,\alpha,\bsgamma_d}:=
 \sum_{\bsh\in\ZZ^d} r_{d,2\alpha,\bsgamma_d} (\bsh) \widehat{f} (\bsh) \overline{\widehat{g}(\bsh)}.
\]
Hence, the inner product and norm in $\calH_{\kor,d,\alpha,\bsgamma_d}$ are the weighted analogues 
of the inner product and norm in $\calH_{\kor,d,\alpha}$. Note that it is necessary to have the index $d$ in the notation of 
the weights $\bsgamma_d$, as in general the set of weights may be different for different choices of $d$.  

\medskip

The space $\calH_{\kor,d,\alpha,\bsgamma}$ is a subspace of $L_2 ([0,1]^d)$, so it makes sense to study 
again $L_2$-approximation for elements of the weighted Korobov space. That is, we then formally consider 
the sequence of weighted operators $S_{d,\bsgamma_d}: \calH_{\kor,d,\alpha,\bsgamma_d}\to L_2 ([0,1]^d)$, 
$d\in\NN$, where $S_{d,\bsgamma_d}(f)=\EMB_{d,\bsgamma_d} (f)=f$. Again, the information complexity can be characterized by the eigenvalues of 
the self-adjoint operators $S_{d,\bsgamma_d}^* S_{d,\bsgamma_d}$. One may ask why it is necessary to 
define the weighted operator $S_{d,\bsgamma_d}$ when it actually coincides with the corresponding embedding 
operator for the unweighted problem; the reason for this formal distinction is that, 
as the norm in $\calH_{\kor,d,\alpha,\bsgamma_d}$ is not the same as in $\calH_{\kor,d,\alpha}$, also the 
eigenvalues of the self-adjoint operators change. 

As shown in \cite{NW08},
the eigenvalues of $S_{d,\bsgamma_d}^* S_{d,\bsgamma_d}$, let us again call them $\lambda_{n,d}$, are 
\[
 \left\{\lambda_{n,d}\colon n\in\NN\right\}=
 \left\{(r_{d,2\alpha,\bsgamma_d} (\bsh))^{-1}\colon \bsh\in\ZZ^d\right\}.
\]
Using these eigenvalues, we can again employ Theorems \ref{thm:ALG_SPT}--\ref{thm:ALG_WT} to determine whether or not certain tractability notions hold for the weighted problem, and this heavily depends on the behavior of the weights. 

Let us illustrate this by an example, where we choose the weights 
as \textit{product weights}. In this case, we have a non-increasing sequence of positive reals, $\gamma_1\ge \gamma_2\ge \cdots >0$, and put
\[
 \gamma_{d,\fraku}=\gamma_{\fraku}:=\prod_{j\in\fraku} \gamma_j,
\]
for any $d\in\NN$ and any $\fraku\subseteq \{1,\ldots,d\}$, where we put $\gamma_{d,\emptyset}=\gamma_{\emptyset}=1$. Note that 
it is justified to neglect the index $d$ in the weights $\bsgamma_{\fraku}$, 
since the sequence of the $\gamma_j$ is chosen independently of $d$. 
For these product weights we get 
\[
 \left\{\lambda_{n,d}\colon n\in\NN\right\}=
 \left\{\prod_{j\in\fraku (\bsh)} \frac{\gamma_j}{\abs{h_j}^{2\alpha}}\colon \bsh\in\ZZ^d\right\}.
\]
Now consider the condition \eqref{eq:cond_SPT}. For $d\in\NN$, $\tau>0$, and $L\in\NN$ we obtain
\begin{eqnarray*}
 \sum_{n=L}^\infty \lambda_{n,d}^\tau &\le& \sum_{n=1}^\infty \lambda_{n,d}^\tau \\
 &=&  \sum_{\bsh\in\ZZ^d} \prod_{j\in\fraku (\bsh)} \frac{\gamma_j^\tau}{\abs{h_j}^{2\alpha\tau}}\\
 &=& \prod_{j=1}^d \left(1 + \sum_{h_j \in\ZZ\setminus \{0\}} 
 \frac{\gamma_j^\tau}{\abs{h_j}^{2\alpha\tau}}\right).
\end{eqnarray*}
If $\tau> 1/(2\alpha)$, we write $\zeta (\cdot)$ to denote the Riemann 
zeta function and obtain
\[
\sum_{n=L}^\infty \lambda_{n,d,\bsgamma_d}^\tau \le
  \prod_{j=1}^d \left(1 + 2\gamma_j^\tau \zeta (2\alpha\tau)\right)\\
  \le \prod_{j=1}^\infty \left(1 + 2\gamma_j^\tau \zeta (2\alpha\tau)\right).
\]
Now we can estimate 
\begin{eqnarray*}
 \prod_{j=1}^\infty \left(1 + 2\gamma_j^\tau \zeta (2\alpha\tau)\right)
 &=& \exp \left[ \sum_{j=1}^\infty \log  \left(1 + 2\gamma_j^\tau \zeta (2\alpha\tau)\right) \right]\\ 
 &\le & 
 \exp \left[\sum_{j=1}^\infty 2\gamma_j^\tau \zeta (2\alpha\tau) \right],
\end{eqnarray*}
where we used that $\log (1+x)\le x$ for $x\ge 0$. 
This shows that summability of the $\gamma_j^\tau$ for $\tau > 1/(2\alpha)$ is a sufficient condition for strong polynomial tractability. 

Indeed, as outlined in \cite[Section 5.3]{NW08}, strong polynomial tractability in this example is equivalent to
\[
 p_{\bsgamma}:= \inf \left\{p\ge 0\colon \sum_{j=1}^\infty \gamma_j^p < \infty\right\}<\infty.
\]
Furthermore, the exponent of SPT is then given by $2\max  \{1/(2\alpha),p_{\bsgamma}\}$. As pointed out in \cite{WW99} and also \cite{NW08}, for the product weights considered here, we even have equivalence of SPT and PT. Furthermore, it is known that WT holds if and 
only if $\inf_{j\ge 1} \gamma_j < 1$. We refer to \cite{EP21} for an overview 
of results corresponding to the present example; for an overview also 
containing results on $L_\infty$-approximation, we refer to \cite{EKP22} and the references therein. 
Moreover, we remark 
that for product weights also the weighted Korobov space has a tensor product structure; indeed, it is the tensor product of the spaces
$\calH_{\kor,1,\alpha,\gamma_j}$, $j\in\{1,\ldots,d\}$. Thus, one could also study tractability in weighted spaces using this structure, which, 
however, we do not need in the present paper.

For results on function approximation in Korobov spaces using information from $\Lambda^{\rm std}$ instead of
$\Lambda^{\rm all}$, we exemplary refer to \cite{DKP22, KSW06, NSW04} and the references therein.

The subject of numerical integration in weighted Korobov spaces has been addressed in a huge number of papers and books. 
The analysis of this question is obviously different from the problem of function approximation, in particular as 
one is limited to information from $\Lambda^{\rm std}$. Generally speaking, also in the case of numerical integration, 
suitably fast decaying weights can yield tractability. Positive results are often obtained by considering so-called \textit{lattice rules} 
as concrete integration rules. We refer to \cite{DKP22}, \cite{DKS13}, and \cite{NW10} for overviews on 
this subject.

\section{Exponential tractability}\label{sec:exp_tractability}

Up to now, we have considered different notions of 
tractability that are defined in terms of a relation between $n_{\CRI}(\varepsilon,S_d)$ and some powers of $d$ and
$\varepsilon^{-1}$. This is what is nowadays sometimes called algebraic (ALG) tractability, and we will also use this term for the rest of the paper. On the other hand, a relatively recent stream of work defines different notions
of tractability in terms of a relation between $n_\CRI (\varepsilon,S_d)$ and some powers of $d$ and $1+\log \varepsilon^{-1}$. Then, the complexity of the problem increases only logarithmically as the error tolerance vanishes. This situation is referred to as exponential
(EXP) tractability, which is the subject of the present section.

Whereas algebraic tractability is usually obtained, e.g., for dealing with classes of functions 
with finite smoothness, a typical setting where we consider exponential tractability is when we deal with 
classes of functions which are 
at least $C^\infty$ or analytic
functions. Indeed, the study of multivariate problems for spaces of infinitely smooth functions motivated 
the introduction of exponential tractability.
Motivating examples with positive exponential
tractability results may be found for specific spaces 
in the papers \cite{CW17, DLPW11, IKLP15, KPW14, KPW17, KMU16, KSU15, SW18} and the references therein. 

If we consider exponential tractability, we can transfer all algebraic tractability notions to the exponential case; we illustrate this by the following example. In \eqref{eq:def_PT}, we defined algebraic polynomial tractability (ALG-PT). Accordingly, we now say that we have \textit{exponential polynomial tractability (EXP-PT)} in the setting $\CRI \in \{\ABS, \NOR\}$  if there exist absolute constants $C,p,q\ge 0$ 
such that
\begin{equation}\label{eq:def_EXP_PT}
n_{\CRI} (\varepsilon, S_d) \le C\, d^q\, \left((1+\log (\varepsilon^{-1})\right)^{-p}\quad \forall \varepsilon\in (0,1), \forall d\in\NN.
\end{equation}
If \eqref{eq:def_EXP_PT} even holds for $q=0$, we speak of \textit{exponential strong polynomial tractability (EXP-SPT)}. If EXP-SPT holds, i.e., if \eqref{eq:def_EXP_PT} holds with $q=0$, then the infimum of the $p$ for which this is the case is called the \textit{exponent of EXP-SPT}.

All other tractability notions, as for instance exponential weak tractability or exponential quasi-polynomial tractability, can be defined analogously. 

We also remark that exponential weak tractability (EXP-WT) is a notion of tractability that expresses the information complexity with respect to the bits of 
accuracy in an approximation problem, as pointed out in \cite{PP14, PPW17}. 

Now, let us again assume that
$\{\calF_d\}_{d\in\NN}$ and $\{\calG_d\}_{d\in\NN}$ are Hilbert spaces, and 
that the $\{S_d : \calF_d \to \calG_d\}_{d \in \NN}$ are
compact linear solution operators. Then, \eqref{eq:infcomp_Hilbert} and \eqref{eq:infcomp_Hilbert_NOR} still apply, but the conditions on the decay 
of the $\lambda_{n,d}$ will change when one switches from algebraic to exponential tractability. 

We exemplary state two theorems, the proofs of which can be found in \cite{KW19}.
\begin{theorem}[Kritzer, Wo\'{z}niakowski]\label{thm:EXP_SPT}
 Let $\{\calF_d\}_{d\in\NN}$ and $\{\calG_d\}_{d\in\NN}$ be Hilbert spaces, and let 
$\{S_d\colon\calF_d\to \calG_d\}_{d\in\NN}$ be compact linear operators. Consider information from $\Lambda^{\rm all}$ and 
the absolute worst case setting on the unit balls $\calB_d$ of the $\calF_d$. 

Then we have EXP-SPT if and only if there exists a constant $\tau>0$ and a constant $L\in\NN$ such that
\begin{equation}\label{eq:cond_EXP_SPT}
 \sup_{d\in\NN} \left(\sum_{n=L}^\infty \lambda_{n,d}^{n^{-\tau}} \right)^{1/\tau} <\infty. 
\end{equation}
The exponent of EXP-SPT is given by
\[
 p^*:=\inf\{1/\tau\colon \tau \ \mbox{satisfies \eqref{eq:cond_EXP_SPT}}\}.
\]
\end{theorem}
\begin{theorem}[Kritzer, Wo\'{z}niakowski]\label{thm:EXP_PT}
 Let $\{\calF_d\}_{d\in\NN}$ and $\{\calG_d\}_{d\in\NN}$ be Hilbert spaces, and let 
$\{S_d\colon\calF_d\to \calG_d\}_{d\in\NN}$ be compact linear operators. Consider information from $\Lambda^{\rm all}$ and 
the absolute worst case setting on the unit balls $\calB_d$ of the $\calF_d$. 

Then we have EXP-PT if and only if there exist constants $\tau_1,\tau_2\ge 0$ and $\tau_3,H>0$ such that
\begin{equation}\label{eq:cond_EXP_PT}
 \sup_{d\in\NN}\, d^{-\tau_1}\, \left(\sum_{n=\lceil H d^{\tau_2}\rceil}^\infty \lambda_{n,d}^{n^{-\tau_3}} \right)^{1/\tau_3} <\infty. 
\end{equation}
\end{theorem}

\medskip

Many further theoretical results on exponential tractability, partly in the context of concrete problem settings, can be found in, e.g.,
\cite{DKPW14, DLPW11, IKLP15, IKPW16a, IKPW16b, KPW14, KPW17, LX16a, LX16b, LX17, PP14, PPW17, SW18, S17, X15}. For results regarding exponential tractability in the tensor product Hilbert case setting, we refer to \cite{HKW20}.

Before we move on to generalized tractability in Section \ref{sec:gen_tractability}, we would like 
to return to our example of Korobov spaces in the next section. 

\section{Example: problems on Korobov spaces, Part 3}\label{sec:example_3} 

Let us return to the (weighted) Korobov spaces considered in Sections \ref{sec:example} and \ref{sec:example_2}. However, we will 
re-define the inner product and norm such that the space consists of (real) analytic functions, which fits the subject 
of exponential tractability. 

Indeed, we define the analytic weighted Korobov space $\calH_{\kor,d,\bsa,\bsb}^{\rm smooth}$ as follows. 
Let $\bsa=\{a_j\}_{j \ge 1}$ and
$\bsb=\{b_j\}_{j \ge 1}$ be two sequences of real positive weights
such that
\begin{equation}\label{eq:cond_a_b}
b_\ast:=\inf_j b_j> 0\ \ \ \ \ \mbox{and}\ \ \ \ \
a_\ast:=\inf_j a_j>0.
\end{equation}
We assume, without loss of generality, that 
$$
a_1 \le a_2 \le a_3 \le \cdots,
$$
i.e., $a_\ast=a_1$.
Fix $\omega\in(0,1)$, and write 
\[
r_{d,\bsa,\bsb} (\bsh):=\omega^{-\sum_{j=1}^{d}a_j \abs{h_j}^{b_j}}
\qquad\mbox{for}\qquad \bsh=(h_1,h_2,\dots,h_d)\in\ZZ^d.
\]

We define $\calH_{\kor,d,\bsa,\bsb}^{\rm smooth}$ as the space of all one-periodic functions 
$f$ with absolutely convergent Fourier series, 
and with finite norm $\norm{f}_{\kor,d,\bsa, \bsb}:=\langle f, f\rangle_{\kor,d,\bsa,\bsb}^{1/2}$,
where the inner product is given by
\[
 \langle f, g\rangle_{\kor,d,\bsa,\bsb}:=
 \sum_{\bsh\in\ZZ^d} r_{d,\bsa,\bsb} (\bsh) \widehat{f} (\bsh) \overline{\widehat{g}(\bsh)}.
\]
Let us again study $L_2$-approximation, i.e., $S_{d,\bsa,\bsb}:\calH_{\kor,d,\bsa,\bsb}^{\rm smooth}\to L_2([0,1]^d)$, $S_{d,\bsa,\bsb} (f)=\EMB_{d,\bsa,\bsb} (f)=f$, again with
$\Lambda^{\rm all}$ as the information class. Also for this problem, it is easy to check that the initial error equals one, so the problem is normalized.
Similarly to Section \ref{sec:example_2}, one can show that the eigenvalues of $\EMB_{d,\bsa,\bsb}^* \EMB_{d,\bsa,\bsb}$ are given by
\begin{equation}\label{eq:eigenvalues_analytic}
 \left\{\lambda_{n,d}\colon n\in\NN\right\}=\left\{\left(r_{d,\bsa,\bsb} (\bsh)\right)^{-1}\colon \bsh\in\ZZ^d\right\}
 =\left\{\omega^{\sum_{j=1}^{d}a_j \abs{h_j}^{b_j}}\colon \bsh\in\ZZ^d\right\}.
\end{equation}
As shown in \cite{DKPW14}, the $n$-th minimal error of $L_2$-approximation on $\calH_{\kor,d,\bsa,\bsb}^{\rm smooth}$ always 
converges exponentially. For this reason, it is justified to study the notions of exponential tractability rather than the 
notions of algebraic tractability  in this context. We could then use Theorems \ref{thm:EXP_SPT} and \ref{thm:EXP_PT}, for example, 
and apply them to the eigenvalues $\left(r_{d,\bsa,\bsb} (\bsh)\right)^{-1}$, in order to analyze the presence 
of EXP-SPT and EXP-PT (for further tractability notions like EXP-WT, we can use corresponding theorems in \cite{KW19}). 
If, for instance, we would like to find out whether $L_2$-approximation on $\calH_{\kor,d,\bsa,\bsb}^{\rm smooth}$ satisfies 
EXP-SPT, we would use Theorem \ref{thm:EXP_SPT}. A slight drawback of this theorem, however, is that we need to know everything about the 
order of the eigenvalues, since we need to study summability of $\lambda_{n,d}^{n^{-\tau}}$ in \eqref{eq:cond_EXP_SPT}. 
This is technically rather involved, and was implicitly done in \cite{DKPW14}. We remark that the concise form 
of the condition in Theorem \ref{thm:EXP_SPT} was not yet known when \cite{DKPW14} was written, and the proof idea in \cite{DKPW14} is less 
straightforward than just checking \eqref{eq:cond_EXP_SPT}. However, as we will see below, an alternative condition that is 
easier to check can be obtained using Theorem \ref{thm:T_strong_tractability}, and we will return to this example again later.
For now, let us just state that an equivalent condition to EXP-SPT of $L_2$-approximation on $\calH_{\kor,d,\bsa,\bsb}^{\rm smooth}$ 
is 
\begin{equation}\label{eq:cond_EXP_SPT_Kor}
 \sum_{j=1}^\infty \frac{1}{b_j} <\infty \quad \mbox{and}\quad \alpha^*:=\liminf_{j\to\infty}\frac{\log a_j}{j}>0.
\end{equation}
This means that we have growth conditions on the weight sequences $\bsa$ and $\bsb$ (and through these parameters and the function 
$r_{d,\bsa,\bsb}$ we have a condition on the decay of the Fourier coefficients of the elements in the function space). 

This was, together with results for various other exponential tractability notions, shown in \cite[Theorem 1]{DKPW14}. 
It is also shown there that for $L_2$-approximation all results are the same independently of whether we consider information from 
$\Lambda^{\rm all}$ or $\Lambda^{\rm std}$, which is rather surprising. Moreover, it was shown in \cite{KPW17} that 
almost all conditions remain the same if we consider $L_\infty$-approximation instead of $L_2$-approximation 
for $\Lambda^{\rm all}$ or $\Lambda^{\rm std}$, and also in the case of $L_\infty$-approximation all results 
for $\Lambda^{\rm all}$ and $\Lambda^{\rm std}$ coincide. For results on numerical integration in $\calH_{\kor,d,\bsa,\bsb}^{\rm smooth}$, 
we refer to \cite{KPW14}.

\section{Generalized tractability}\label{sec:gen_tractability}

The previous section on exponential tractability naturally raises the question whether we can define tractability more generally in terms of 
functions of $\varepsilon^{-1}$ and $d$ instead of the special cases of polynomial or exponential functions. Indeed, such an analysis is possible. This was first done by Gnewuch and Wo\'{z}niakowski in a series of papers in which they introduced generalized tractability, see \cite{GW07}--\cite{GW11}, and also \cite[Chapter 8]{NW08}. In these references, the authors provide an in-depth analysis of generalized tractability, mostly for the case of tensor product problems. In the recent paper \cite{EHK23}, generalized tractability was analyzed in the Hilbert space setting (as in Sections \ref{sec:Hilbert} and \ref{sec:exp_tractability}) without the assumption of a tensor product structure.

We will present selected findings on generalized tractability here. Generally speaking, we would like to define tractability in terms of bounds on the information complexity that are represented by a generalized \textit{tractability function} $T$ depending on $\varepsilon^{-1}$ and $d$. In this setting, it is also easily possible to consider the range of $\varepsilon$ not only as $(0,1)$, but more generally as $(0,\infty)$, which gives us more flexibility.

In this section, we again work with the assumption that $\{\calF_d\}_{d\in\NN}$ and $\{\calG_d\}_{d\in\NN}$ are Hilbert spaces, and 
we let $\{S_d : \calF_d \to \calG_d\}_{d \in \NN}$ be a sequence of
compact linear solution operators. Let us assume that there is an infinite number of positive $\lambda_{n,d}$ for every $d\in\NN$. 

We need to assume several properties of the function $T$ (we follow \cite{EHK23} in our notation here, but similar observations were also made before in \cite{GW07}--\cite{GW11}). We fix $s\in\NN$, and define $T$ as a function of three (or, to be more precise, $s+2$) arguments,
\begin{equation} \label{eq:Tspec}
    T :(0,\infty) \times \NN \times [0,\infty)^s \rightarrow (0,\infty).
\end{equation}
The basic idea is that we define our approximation problem to be tractable if $n_{\CRI} (\varepsilon, S_d) \le  C_{\bsp}\, T(\varepsilon^{-1},d,\bsp)$ for some constant $C_{\bsp}$, depending only on the parameter $\bsp$. The parameter $\bsp$ is an $s$-dimensional vector with $s \geq 1$, where we assume that every component of $\bsp$ is nonnegative. We make the following assumptions on $T$. 
\begin{itemize}
 \item $T$ is non-decreasing in all variables, which implies that the problem is expected to become no easier by decreasing $\varepsilon$, or increasing $d$. Furthermore, increasing the components of $\bsp$ allows for a possibly looser bound on the information complexity. 
 \item We also require that 
 \[
	\lim_{\varepsilon \to 0} T(\varepsilon^{-1},d,\bsp) = \infty \qquad \forall d \in \NN, \ \bsp \in [0,\infty)^s,
\]
which makes sense since we assumed that there are an infinite number of positive $\lambda_{n,d}$. 
\item  We require the existence of the following limit,
\[
T(0,d,\bsp):=\lim_{\varepsilon\to\infty}T(\varepsilon^{-1},d,\bsp) = \inf_{\varepsilon\in (0,\infty)} T(\varepsilon^{-1},d,\bsp) \ge T(0,1,\bszero) > 0,
\]
where $\bszero$ denotes the vector consisting only of zeros.
\item We require the existence of a $K_{\bsp,\tau}$ depending on $\bsp$ and $\tau$, but independent of $\varepsilon$ and $d$, such that
\begin{multline*}
	(T(\varepsilon^{-1},d,\bsp))^\tau \le K_{\bsp,\tau} T(\varepsilon^{-1},d,\tau \bsp)\\ \forall \varepsilon \in (0,\infty), \ d \in \NN, \ \bsp\in[0,\infty)^s, \ \tau\in [1,\infty).
\end{multline*}
\end{itemize}

With these assumptions, we can formally define, e.g., \textit{$T$-tractability} and \textit{strong $T$-tractability}. 
Indeed, a problem is called $T$-tractable with parameter $\bsp$ if there exists a positive constant $C_{\bsp}$, 
which is independent of $\varepsilon$ and $d$, such that
\begin{equation} \label{eq:def_T_tractability}
	n_{\CRI} (\varepsilon, S_d) \le C_{\bsp}\, T(\varepsilon^{-1},d,\bsp) \qquad \forall \varepsilon \in (0,\infty), \ \forall d \in \NN.
\end{equation}
A problem is \emph{strongly}
$T$-tractable with parameter $\bsp$ if the information complexity is independent of the dimension of the problem, that is, there exists a positive constant $C_{\bsp}$, again independent of $\varepsilon$ and $d$, such that
\begin{equation} \label{eq:def_T_strong_tractability}
	n_{\CRI} (\varepsilon, S_d) \le C_{\bsp}\, T(\varepsilon^{-1},1,\bsp) \qquad \forall \varepsilon\in (0,\infty) , \ \forall d \in \NN.
\end{equation}
Also for $T$-tractability, one can study the exponents of tractability, which 
is more technical than for tractability in the algebraic or exponential cases. Moreover, also in this case one has a potential trade-off of exponents. We refer to \cite{EHK23} for details and results on exponents.

As mentioned above, we can consider special choices of the function $T$ and obtain examples that we have 
seen in the previous sections. If we would like to consider ALG-PT as in \eqref{eq:def_PT}, we 
would choose $\bsp=(q,p)$ and $T(\varepsilon^{-1},d,\bsp)=d^q\varepsilon^{-p}$ for $\varepsilon\in (0,1)$. Alternatively, we
can allow the wider range $\varepsilon\in (0,\infty)$, and would then choose 
\[
 T(\varepsilon^{-1},d,\bsp)=d^q\max\{1,\varepsilon^{-1}\}^p \qquad \forall \varepsilon \in (0,\infty) , \ \forall d \in \NN,
\]
where the maximum is used to cover the cases where $\varepsilon\ge 1$. If we would like to consider EXP-PT for $\varepsilon\in (0,\infty)$, we would again 
choose $\bsp=(q,p)$ and
\[
 T(\varepsilon^{-1},d,\bsp)=\left(1+\log(\max\{1,\varepsilon^{-1}\})\right)^p\, d^q \qquad \forall \varepsilon \in (0,\infty) , \ \forall d \in \NN.
\]

Let us state an exemplary result on $T$-tractability from \cite{EHK23}. 
\begin{theorem}[Emenike, Kritzer, Hickernell]\label{thm:T_strong_tractability}
Let $\{\calF_d\}_{d\in\NN}$ and $\{\calG_d\}_{d\in\NN}$ be Hilbert spaces, and let 
$\{S_d\colon\calF_d\to \calG_d\}_{d\in\NN}$ be compact linear operators. Consider information from $\Lambda^{\rm all}$ and 
the absolute worst case setting on the unit balls $\calB_d$ of the $\calF_d$. Furthermore, 
let $T$ be a tractability function fulfilling all assumptions as above.  

Then we have $T$-tractability if and only if there exists a 
$\bsp \in [0,\infty)^s$ and an integer $H > 0$ such that
\[
     \sup_{d \in \NN} \sum_{n =\lceil H\cdot T(0,d,\bsp)\rceil}^\infty \frac{1}{T(\lambda_{n,d}^{-1/2},1,\bsp)} < \infty.
\]
\end{theorem}
Note that, if we choose $T(\varepsilon^{-1},d,\bsp)=\varepsilon^{-p}\, d^q$ for $\varepsilon\in (0,1)$ and $d\in\NN$, 
we recover Theorem \ref{thm:ALG_PT} from Theorem \ref{thm:T_strong_tractability}. Note also that, 
if we choose  $T(\varepsilon^{-1},d,\bsp)=\left(1+\log(\max\{1,\varepsilon^{-1}\})\right)^p\, d^q$, 
we do not precisely recover Theorem \ref{thm:EXP_PT}, but obtain a slightly modified condition, which is actually easier 
to check than that in Theorem \ref{thm:EXP_PT}. It is shown in \cite{EHK23} that these conditions are equivalent.

We would like to highlight another choice of the function $T$, which was first analyzed in \cite{GW11},
and gives rise to the concept of (algebraic) \textit{quasi-polynomial tractability (ALG-QPT)}. The idea of QPT was introduced by the authors of \cite{GW11} 
as an example of the choice of the tractability function $T$ that guarantees 
slightly more than polynomial growth for increasing $\varepsilon^{-1}$ or $d$, but still at a manageable rate, which explains the name ``quasi-polynomial'' (we refer to \cite[p.~313]{GW11} for a more detailed discussion). 
Indeed, following \cite{GW11}, for ALG-QPT, we choose $\bsp$ as a nonnegative scalar $p$, and 
\begin{multline*}
 T(\varepsilon^{-1},d,\bsp)= \exp\left(p (1+\log d) (1+ \log(\max\{1,\varepsilon^{-1}\}))\right)\\ 
\forall \varepsilon \in (0,\infty) , \ \forall d \in \NN.
\end{multline*}
Since its introduction in \cite{GW11}, ALG-QPT, but also its ``exponential'' 
counterpart EXP-QPT, have been studied in a large number of research 
papers, and also in \cite{NW12}. 

To conclude this section, we mention yet another way in which the definition 
of tractability can be generalized: also this generalization was laid out 
in the papers \cite{GW07}--\cite{GW11}, and studied there in detail, in 
particular for linear tensor product problems. Indeed, one may ask under which circumstances tractability holds when one does not allow 
all of the range $(0,\infty)\times \NN$ for $(\varepsilon^{-1},d)$, but a 
somewhat restricted domain, as for example 
\[
 \{(\varepsilon^{-1},d) : \varepsilon \in (V(d),\infty), \ d \in \NN \}, 
\]
where $V:\NN\to [0,\infty)$ is a suitably chosen function. Such a scenario 
requires a more careful analysis of the interplay between the restricted domain and $T$, and of the decay of 
the eigenvalues $\lambda_{n,d}$ of the problem. We refer to 
\cite{GW07}--\cite{GW11} and also \cite{NW08} and \cite{EHK23} for results.

\section{Example: problems on Korobov spaces, Part 4}\label{sec:example_4} 

Let us, finally, return once more to the problem of $L_2$-approximation of functions in 
the analytic Korobov space $\calH_{\kor,d,\bsa,\bsb}^{\rm smooth}$. The eigenvalues of 
the problem are given in \eqref{eq:eigenvalues_analytic}. We will  show how we can 
use Theorem \ref{thm:T_strong_tractability} with the special choice 
$T(\varepsilon^{-1},d,\bsp)=\left(1+\log(\max\{1,\varepsilon^{-1}\})\right)^p\, d^q$ and $\bsp=(p,q)$ to 
obtain a sufficient condition for EXP-SPT. 
Indeed, consider the expression 
\[
 \sum_{n=1}^\infty \frac{1}{T\left(\lambda_{n,d}^{-1/2},1,(p,0)\right)}
\]
for a real $p>0$. 
Inserting our concrete choice of $T$ from above, we obtain
\begin{eqnarray*}
 \sum_{n=1}^\infty \frac{1}{T\left(\lambda_{n,d}^{-1/2},1,(p,0)\right)}
 &=& 
 \sum_{\bsh\in\ZZ^d} \frac{1}{\left(1+ \log \left(\omega^{-\frac{1}{2}\sum_{j=1}^d a_j \abs{h_j}^{b_j}}\right)\right)^p}\\
  &=& 
 1+\sum_{\bsh\in\ZZ^d\setminus \{\bszero\}} \frac{1}{\left(1+ \log \left(\omega^{-\frac{1}{2}\sum_{j=1}^d a_j \abs{h_j}^{b_j}}\right)\right)^p}\\ 
  &\le& 
 1 + \sum_{\bsh\in\ZZ^d\setminus \{\bszero\}} \frac{1}{\left(\log \left(\omega^{-\frac{1}{2}\sum_{j=1}^d a_j \abs{h_j}^{b_j}}\right)\right)^p}\\
 &=&
 1+\frac{2^p}{(\log (\omega^{-1}))^p} \sum_{\bsh\in\ZZ^d\setminus \{\bszero\}} 
 \frac{1}{\left(\sum_{j=1}^d a_j \abs{h_j}^{b_j}\right)^p}.
\end{eqnarray*}
Suppose that \eqref{eq:cond_EXP_SPT_Kor} holds, and let us study the sum 
\[
 \Sigma:=\sum_{\bsh\in\ZZ^d \setminus \{\bszero\}} 
 \frac{1}{\left(\sum_{j=1}^d a_j \abs{h_j}^{b_j}\right)^p}.
\]
Note that, for any $\bsh\in \ZZ^d \setminus \{\bszero\}$, we have 
$\sum_{j=1}^d a_j \abs{h_j}^{b_j} \ge a_1\ge \lfloor a_1 \rfloor \ge 0$, 
and that in any case we have 
$\sum_{j=1}^d a_j \abs{h_j}^{b_j} \ge a_1>0$ as $a_1=a_*$ is assumed to be 
strictly positive (see~\eqref{eq:cond_a_b}).

Hence we obtain
\begin{eqnarray*}
 \Sigma &=& \sum_{\substack{\bsh\in\ZZ^d \setminus \{\bszero\} \\
 a_1\, \le\, \sum_{j=1}^d a_j \abs{h_j}^{b_j} \,<\, \lfloor a_1 \rfloor +1\\ }}
 \frac{1}{\left(\sum_{j=1}^d a_j \abs{h_j}^{b_j}\right)^p}\\
 &&+\sum_{\ell=\lfloor a_1 \rfloor +1}^\infty \sum_{\substack{\bsh\in\ZZ^d \setminus \{\bszero\} \\
 \ell\, \le\, \sum_{j=1}^d a_j \abs{h_j}^{b_j} \,<\, \ell +1\\ }}
 \frac{1}{\left(\sum_{j=1}^d a_j \abs{h_j}^{b_j}\right)^p}\\
 &\le& 
 \sum_{\substack{\bsh\in\ZZ^d \setminus \{\bszero\} \\
 a_1\, \le\, \sum_{j=1}^d a_j \abs{h_j}^{b_j} \,<\, \lfloor a_1 \rfloor +1\\ }}
 \frac{1}{a_1^p}\ 
 +\sum_{\ell=\lfloor a_1 \rfloor + 1}^\infty \sum_{\substack{\bsh\in\ZZ^d \setminus \{\bszero\} \\
 \ell\, \le\, \sum_{j=1}^d a_j \abs{h_j}^{b_j} \,<\, \ell +1\\ }}
 \frac{1}{\ell^p}\\
 &=& \frac{1}{a_1^p}\, \abs{\left\{\bsh\in\ZZ^d \setminus \{\bszero\} \colon
a_1\, \le\, \sum_{j=1}^d a_j \abs{h_j}^{b_j} \,<\, \lfloor a_1 \rfloor +1\right\}}\\
 &&+\sum_{\ell=\lfloor a_1 \rfloor + 1}^\infty \frac{1}{\ell^p}
  \abs{\left\{\bsh\in\ZZ^d \setminus \{\bszero\} \colon
\ell\, \le\, \sum_{j=1}^d a_j \abs{h_j}^{b_j} \,<\, \ell +1\right\}}\\
 &\le& \sum_{\ell=\lfloor a_1 \rfloor}^\infty \frac{1}{\max\{a_1, \ell\}^p}
  \abs{\left\{\bsh\in\ZZ^d \colon
\sum_{j=1}^d a_j \abs{h_j}^{b_j} \,<\, \ell +1\right\}}.
\end{eqnarray*}
By the condition on the sequence $\bsa$ in \eqref{eq:cond_EXP_SPT_Kor}, we know that 
for any $\delta\in (0,\alpha^*)$ there exists a $j_\delta^*$ such that $a_j \ge e^{\delta j}$ 
for all $j\ge j_\delta^*$. Since 
$\ell +1 > a_1$ if $\ell\ge \lfloor a_1 \rfloor$, we can now 
make use of an estimate in the proof of Theorem 1 in \cite{DKPW14}. This estimate states that
\[
 \abs{\left\{\bsh\in\ZZ^d \colon
\sum_{j=1}^d a_j \abs{h_j}^{b_j} \,<\, \ell +1\right\}}
\le 3^{j_\delta^*} \,a_1^{-B}\, (\ell+1)^{B+ (\log 3)/\delta}.
\]
Using the latter estimate, we obtain
\[
 \Sigma \le 3^{j_\delta^*} \,a_1^{-B}\, \sum_{\ell=\lfloor a_1 \rfloor}^\infty \frac{(\ell+1)^{B+ (\log 3)/\delta}}{\max\{a_1,\ell\}^p}, 
\]
which is finite and independent of $d$ if we choose $p> B+ (\log 3)/\delta +1$. Thus we obtain 
EXP-SPT by applying Theorem \ref{thm:T_strong_tractability}, and we can choose $p$ arbitrarily close to
$B+ (\log 3)/\alpha^* +1$ with this method. 

We remark that the slightly more complicated original approach to showing this result in \cite{DKPW14} 
implies a better bound on the exponent $p^*$ of EXP-SPT, namely 
\[
 \max\left\{B, \frac{\log 3}{\alpha^*}\le p^* \le B + \frac{\log 3}{\alpha^*}\right\}.
\]
Nevertheless, our example shows that Theorem \ref{thm:T_strong_tractability} is useful in 
determining that EXP-SPT holds for our approximation problem.

\section{Conclusion}

In this paper, we have summarized several classical and newer results on tractability analysis. 
As the title of the paper suggests, this article does not claim to give by any means a full account 
of the field (which would be virtually impossible; recall that there exists the three-volume book \cite{NW08}--\cite{NW12} on this subject). 
Nevertheless, we hope to have provided some idea of how the theory works for a particular type of problems, with 
a special focus on the more recent developments regarding exponential and generalized tractability.

\section*{Acknowledgements}

The author would like to thank David Krieg and Friedrich Pillichshammer for valuable discussions, and 
three anonymous referees for helpful comments. 
The author is supported by the Austrian Science Fund (FWF), Project P34808. 
For the purpose of open access, the author has applied a CC BY public copyright licence to any author accepted manuscript version arising from this submission.
The paper was initiated during the Dagstuhl Seminar 23351 `Algorithms and Complexity for Continuous Problems', in Schloss Dagstuhl, Wadern, Germany, in August 2023.
We are grateful to the Leibniz Center Schloss Dagstuhl.

\begin{small}
	\noindent\textbf{Author's address:}\\

	\noindent Peter Kritzer\\
	Johann Radon Institute for Computational and Applied Mathematics (RICAM)\\
	Austrian Academy of Sciences\\
	Altenbergerstr. 69, 4040 Linz, Austria.\\
	\texttt{peter.kritzer@oeaw.ac.at}

\end{small}

\end{document}